\documentclass[a4paper,12pt,reqno]{amsart}
\usepackage{amsfonts, color}
\usepackage{amsmath}
\usepackage{amssymb}
\usepackage[a4paper]{geometry}
\usepackage{mathrsfs}
\usepackage[colorlinks]{hyperref}
\usepackage{hyperref}
\renewcommand\eqref[1]{(\ref{#1})}

\makeatletter
\newcommand*{\mint}[1]{%
  \mint@l{#1}{}%
}
\newcommand*{\mint@l}[2]{%
  \@ifnextchar\limits{%
    \mint@l{#1}%
  }{%
    \@ifnextchar\nolimits{%
      \mint@l{#1}%
    }{%
      \@ifnextchar\displaylimits{%
        \mint@l{#1}%
      }{%
        \mint@s{#2}{#1}%
      }%
    }%
  }%
}
\newcommand*{\mint@s}[2]{%
  \@ifnextchar_{%
    \mint@sub{#1}{#2}%
  }{%
    \@ifnextchar^{%
      \mint@sup{#1}{#2}%
    }{%
      \mint@{#1}{#2}{}{}%
    }%
  }%
}
\def\mint@sub#1#2_#3{%
  \@ifnextchar^{%
    \mint@sub@sup{#1}{#2}{#3}%
  }{%
    \mint@{#1}{#2}{#3}{}%
  }%
}
\def\mint@sup#1#2^#3{%
  \@ifnextchar_{%
    \mint@sup@sub{#1}{#2}{#3}%
  }{%
    \mint@{#1}{#2}{}{#3}%
  }%
}
\def\mint@sub@sup#1#2#3^#4{%
  \mint@{#1}{#2}{#3}{#4}%
}
\def\mint@sup@sub#1#2#3_#4{%
  \mint@{#1}{#2}{#4}{#3}%
}
\newcommand*{\mint@}[4]{%
  \mathop{}%
  \mkern-\thinmuskip
  \mathchoice{%
    \mint@@{#1}{#2}{#3}{#4}%
        \displaystyle\textstyle\scriptstyle
  }{%
    \mint@@{#1}{#2}{#3}{#4}%
        \textstyle\scriptstyle\scriptstyle
  }{%
    \mint@@{#1}{#2}{#3}{#4}%
        \scriptstyle\scriptscriptstyle\scriptscriptstyle
  }{%
    \mint@@{#1}{#2}{#3}{#4}%
        \scriptscriptstyle\scriptscriptstyle\scriptscriptstyle
  }%
  \mkern-\thinmuskip
  \int#1%
  \ifx\\#3\\\else_{#3}\fi
  \ifx\\#4\\\else^{#4}\fi
}
\newcommand*{\mint@@}[7]{%
  \begingroup
    \sbox0{$#5\int\m@th$}%
    \sbox2{$#5\int_{}\m@th$}%
    \dimen2=\wd0 %
    \let\mint@limits=#1\relax
    \ifx\mint@limits\relax
      \sbox4{$#5\int_{\kern1sp}^{\kern1sp}\m@th$}%
      \ifdim\wd4>\wd2 %
        \let\mint@limits=\nolimits
      \else
        \let\mint@limits=\limits
      \fi
    \fi
    \ifx\mint@limits\displaylimits
      \ifx#5\displaystyle
        \let\mint@limits=\limits
      \fi
    \fi
    \ifx\mint@limits\limits
      \sbox0{$#7#3\m@th$}%
      \sbox2{$#7#4\m@th$}%
      \ifdim\wd0>\dimen2 %
        \dimen2=\wd0 %
      \fi
      \ifdim\wd2>\dimen2 %
        \dimen2=\wd2 %
      \fi
    \fi
    \rlap{%
      $#5%
        \vcenter{%
          \hbox to\dimen2{%
            \hss
            $#6{#2}\m@th$%
            \hss
          }%
        }%
      $%
    }%
  \endgroup
}
\numberwithin{equation}{section}
\theoremstyle{plain}
\newtheorem{thm}{Theorem}[section]

\newtheorem{cor}[thm]{Corollary}

\theoremstyle{definition}
\newtheorem{defn}[thm]{Definition}
\newtheorem{rem}[thm]{Remark}

\title[Heat equation for Sturm-Liouville operator with singular]{Heat equation for Sturm-Liouville operator with singular propagation and potential}
\author[M. Ruzhansky]{Michael Ruzhansky}
\address{
  Michael Ruzhansky:
  \endgraf
  Department of Mathematics: Analysis, Logic and Discrete Mathematics
  \endgraf
  Ghent University, Belgium
  \endgraf
 and
  \endgraf
  School of Mathematical Sciences
  \endgraf
  Queen Mary University of London
  \endgraf
  United Kingdom
  \endgraf
  {\it E-mail address} {\rm michael.ruzhansky@ugent.be}
  }

\author[A. Yeskermessuly]{Alibek Yeskermessuly}
\address{
  Alibek Yeskermessuly:
  \endgraf
  Altynsarin Arkalyk Pedagogical Institute,  \\ 
  \endgraf
  Arkalyk, Kazakhstan
  \endgraf
and 
  \endgraf
  Department of Mathematics: Analysis, Logic and Discrete Mathematics
  \endgraf
  Ghent University, Belgium
  \endgraf
  {\it E-mail address} {\rm alibek.yeskermessuly@gmail.com}}

\begin{document}

\thanks{The authors are supported by the FWO Odysseus 1 grant G.0H94.18N: Analysis and Partial Differential Equations and by the Methusalem programme of the Ghent University Special Research Fund (BOF) (Grant number 01M01021). Michael Ruzhansky is also supported by EPSRC grants EP/R003025/2 and EP/V005529/1.
}

%
%

%
%

%
\maketitle              

\begin{abstract} 
This article considers the initial boundary value problem for the heat equation with the time-dependent Sturm-Liouville operator with singular potentials. To obtain a solution by the method of separation of variables, the problem is reduced to the problem of eigenvalues of the Sturm-Liouville operator. Further on, the solution to the initial boundary value problem is constructed in the form of a Fourier series expansion. A heterogeneous case is also considered. Finally, we establish the well-posedness of the equation in the case when the potential and initial data are distributions, also for singular time-dependent coefficients.

\end{abstract}

\section{Introduction}

The aim of this paper is to establish the well-posedness results for the heat equation with the Sturm-Liouville operator containing singular potentials. When confronted with issues involving strong singularities, a concept known as \textquotedblleft very weak solutions\textquotedblright \, was introduced in a prior work by Garetto and the first author (\cite{Gar-Ruz}). The rationale behind this approach is that when the equation contains products of various terms, it no longer possesses a clear definition in distribution spaces. Consequently, an alternative perspective on establishing the well-posedness of the equation becomes necessary.

Further establishment of very weak solutions for various types of problems was considered in the works \cite{ARST1}, \cite{ARST2}, \cite{ARST3}, \cite{CRT1}, \cite{CRT2}, \cite{CRT0}, \cite{Lan-Ham}, \cite{Ruz-Tok}, \cite{R-Y}. In \cite{RYS} and \cite{RYes} very weak solutions of the wave equation for the Sturm-Liouville operator with singular potentials were obtained in bounded domains. It turned out that the methods proposed in papers \cite{RYS} and \cite{RYes} can be extended for the heat equations, also allowing singular time-dependent coefficients.

It is well known that the heat equation is easily reduced to ordinary linear equations by the method of separation of variables (see, for example, \cite{Separ}) and the consideration of the eigenvalue problem of the Sturm-Liouville operator.

To obtain the main results, it will be useful to give some preliminary information about the Sturm-Liouville operator with singular potentials. In the paper of Savchuk and Shkalikov \cite{Sav-Shk}, asymptotic estimates of the eigenvalues and eigenfunctions of the Sturm-Liouville operator with singular potentials were obtained. This method was later developed in the works of \cite{N-zSk}, \cite{Savch}, \cite{Sav-Shk2}, \cite{SV}. To apply the concept of very weak solutions to our problem, we will first consider cases where the potentials of the Sturm-Liouville operator are more regular.

More precisely, consider the Sturm-Liouville operator $\mathcal{L}$ generated in the interval (0,1) by the differential expression
\begin{equation}\label{St-L}
    \mathcal{L}y:=-\frac{d^2}{dx^2}y+q(x)y,
\end{equation}
with the boundary conditions
\begin{equation}\label{Dirihle}
    y(0)=y(1)=0. 
\end{equation}
We will first consider the potential $q$ such that
\begin{equation}\label{con-q}
    q(x)=\nu'(x), \qquad \nu\in L^2(0,1).
\end{equation}

Introducing the quasi-derivative in the form
$$y^{[1]}(x)=y'(x)-\nu(x)y(x),$$
the eigenvalue equation $\mathcal{L}y=\lambda y$ reduces to the system
$$\left(\begin{array}{c}
     \phi_1  \\
     \phi_2 
\end{array}\right)'=\left(\begin{array}{cc}
    \nu & 1 \\
    -\lambda-\nu^2 & -\nu
\end{array}\right)\left(\begin{array}{c}
     \phi_1  \\
     \phi_2
\end{array}\right),\quad \phi_1(x,\lambda)=y(x,\lambda),\,\,\phi_2(x,\lambda)=y^{[1]}(x,\lambda).$$

We apply he substitution
$$\phi_1(x,\lambda)=r(x,\lambda)\sin \theta(x,\lambda),\qquad \phi_2(x,\lambda)=\lambda^\frac{1}{2}r(x,\lambda)\cos \theta(x,\lambda),$$
which is a modification of the Prufer substitution (\cite{Ince}), where
\begin{equation}\label{theta}
\theta'(x,\lambda)=\lambda^\frac{1}{2}+\lambda^{-\frac{1}{2}}\nu^2(x)\sin^2 \theta(x,\lambda)+\nu(x)\sin 2\theta(x,\lambda),
\end{equation}
\begin{equation}\label{r}
    r'(x,\lambda)=-r(x,\lambda)\left[\frac{1}{2}\nu^2(x)\lambda^{-\frac{1}{2}}\sin 2\theta(x,\lambda)+\nu(x)\cos 2\theta(x,\lambda)\right].
\end{equation}

We will look for a solution to the equation \eqref{theta} in the form $\theta(x,\lambda)=\lambda^\frac{1}{2} x+\eta(x,\lambda), $ where
\begin{equation}\label{eta}
\eta(x,\lambda)=\lambda^{-\frac{1}{2}}\int\limits_0^x \nu^2(s)\sin^2\theta(s,\lambda)ds+\int\limits_0^x \nu(s)\sin(2\lambda^\frac{1}{2}s+2\eta(s,\lambda))ds
\end{equation}
By using the method of successive approximations, one can  show that this equation has a solution that is uniformly bounded for $0\leq x\leq 1$ and $\lambda\geq 1$. Since $|\nu|^2\in L^1(0,1)$, according to the Riemann-Lebesgue lemma (\cite{Sav-Shk}), $\eta(x,\lambda)=o(1)\to 0$ as $\lambda \to \infty$ uniformly in $x\in (0,1)$. Therefore,
$$\theta(x,\lambda)=\lambda^\frac{1}{2}x+o(1),$$
furthermore, according to \eqref{theta} and \eqref{eta},
$\theta(0,\lambda)=0.$

Solving \eqref{r}, we find
\begin{equation}\label{r(x)}r(x,\lambda)=\exp{\left(-\int\limits_0^x \nu(s)\cos2\theta(s,\lambda)ds-\frac{1}{2}\lambda^{-\frac{1}{2}}\int\limits_0^x \nu^2(s)\sin 2\theta(s,\lambda)ds\right)},
\end{equation}
hence we get $r(0, \lambda)=1$. Using the Riemann-Lebesgue lemma to the last expression, we also have
$$
r(x,\lambda)=\exp{\left(-\int\limits_0^x \nu(s)\cos2\theta(s,\lambda)ds+o(1)\right)}=1+o(1)\quad \text{for }\lambda\to \infty
$$
uniformly in $x\in (0,1)$.

Using the boundary conditions \eqref{Dirihle} we obtain
$$ \phi_1(1,\lambda)=r(1,\lambda)\sin\theta(1,\lambda)=0.$$
According to \eqref{r(x)}, $r(1,\lambda)\neq 0$, therefore $\sin \theta(1,\lambda)=0,\,\,\theta(1,\lambda)=\pi n$.
Then the eigenvalues of the problem $\mathcal{L}y=\lambda y$ with the boundary conditions \eqref{Dirihle} are real (\cite{Al-G}), and given by
\begin{equation}\label{e-val}
    \lambda_n=(\pi n)^2(1+o(n^{-1})),\qquad n=1,2,...,
\end{equation}
and the corresponding eigenfunctions are
\begin{equation}\label{sol-SL}
    \Tilde{\phi}_n(x)=r_n(x)\sin(\sqrt{\lambda_n}x +\eta_n(x)).
\end{equation}
Since $\lambda_n$ are real and according to \eqref{con-q}, $\nu$ is a real valued function, then eigenfunctions $\Tilde{\phi}_n$ are real.

Since $\Tilde{\phi}_n$ are the solutions of $\mathcal{L}y=\lambda y$, substituting 
$$y^{[1]}(x)=\Tilde{\phi}'_n(x)-\nu(x)\Tilde{\phi}_n(x),\,\,
y^{[1]}(x)=\lambda^\frac{1}{2}r_n(x)\cos \theta_n(x),$$ we can define the first derivatives of $\Tilde{\phi}_n$ in the following form
\begin{equation}\label{phi-der}
\Tilde{\phi}'_n(x)=\sqrt{\lambda_n}r_n(x)\cos(\theta_n(x))+\nu(x)\Tilde{\phi}_n(x).
\end{equation}

Now, let us estimate the $\|\Tilde{\phi}_n\|_{L^2}$ using the formula \eqref{sol-SL} as follows
\begin{eqnarray}\label{est-high}
\|\Tilde{\phi}_n\|^2_{L^2}&=&\int\limits_0^1\left|r_n(x)\sin\left(\lambda_n^{\frac{1}{2}}x+\eta_n(x)\right)\right|^2dx\leq \int\limits_0^1\left|r_n(x)\right|^2dx\nonumber\\
&\leq& \int\limits_0^1\left|\exp{\left(-\int\limits_0^x \nu(s)\cos{2\theta_n(s)}ds-\frac{1}{2}\frac{1}{\sqrt{\lambda_n}}\int\limits_0^x\nu^2(s)\sin{2\theta_n(s)}ds\right)}\right|^2dx\nonumber\\
&\lesssim& \int\limits_0^1\exp{\left(2\int\limits_0^x|\nu(s)|ds+\frac{1}{\sqrt{\lambda_n}}\int\limits_0^x|\nu^2(s)|ds\right)}dx\nonumber\\
&\lesssim& \exp{2\left(\|\nu\|_{L^2}+\lambda_n^{-\frac{1}{2}}\|\nu\|^2_{L^2}\right)}<\infty.
\end{eqnarray}

In addition, according Theorem 4 in \cite{Sav-Shk}, we get
\begin{equation}\label{est_low}
  \Tilde{\phi}_n(x)=\sin(\pi nx)+o(1)  
\end{equation}
for sufficiently large $n$, where $o(1)\to 0$ as $n\to \infty$ uniformly in $x\in (0,1)$. Combining this with \eqref{sol-SL}, we see that there exist some $C_0>0$, such that
\begin{equation}\label{low-est}
0<C_0\leq \|\Tilde{\phi}_n\|_{L^2}<\infty \quad \text{for all } n.
\end{equation}

Since the eigenfunctions of the operator $\mathcal{L}$ form an orthogonal basis in $L^2(0,1)$, we can normalize them so that  
\begin{equation}\label{norm-phi}
  \phi_n(x)=\frac{\Tilde{\phi}_n(x)}{\sqrt{\langle \Tilde{\phi}_n,\Tilde{\phi}_n}\rangle}=\frac{\Tilde{\phi}_n(x)}{\|\Tilde{\phi}_n\|_{L^2}}.  
\end{equation}

\section{Homogeneous heat equation}

We consider the heat equation
\begin{equation}\label{C.p1}
         \partial_t u(t,x)+a(t)\mathcal{L} u(t,x)=0,\qquad (t,x)\in [0,T]\times (0,1),
\end{equation}
with initial condition
\begin{equation}\label{C.p2} 
u(0,x)=u_0(x),\,\,\, x\in (0,1),
\end{equation}
and with Dirichlet boundary conditions
\begin{equation}\label{C.p3}
u(t,0)=0=u(t,1),\qquad t\in [0,T],
\end{equation}
where $a(t)\geq a_0>0$ in $[0,T]$ and $\mathcal{L}$ is defined by

\begin{equation}\label{1}
    \mathcal{L} u(t,x):=-\partial^2_x u(t,x)+q(x)u(t,x),\qquad x\in(0,1),
\end{equation}
and $q=\nu'(x)$, $\nu \in L^2(0,1)$. 

To obtain the results of this chapter concerning the initial/boundary value problem \eqref{C.p1}-\eqref{C.p3}, we will first carry out the analysis in the more regular case when $q \in L^ \infty(0,1) $. In this case we obtain the well-posedness in the Sobolev spaces $W^k_\mathcal{L}$ associated with the operator $\mathcal{L}$: we define the Sobolev spaces $W^k_\mathcal{L}$ associated with $\mathcal {L}$, for any $k \in \mathbb{R}$, as the spaces
$$W^k_\mathcal{L}:=\left\{f\in \mathcal{D}'_\mathcal{L}(0,1):\,\mathcal{L}^{k/2}f\in L^2(0,1)\right\},$$
with the norm $\|f\|_{W^k_\mathcal{L}}:=\|\mathcal{L}^{k/2}f\|_{L^2}$. The global space of distributions $\mathcal{D}'_\mathcal{L}(0,1)$ is defined as follows.

The space $C^\infty_\mathcal{L}(0,1):=\mathrm{Dom}(\mathcal{L}^\infty)$ is called the space of test functions for $\mathcal{L}$, where we define 
$$\mathrm{Dom}(\mathcal{L}^\infty):=\bigcap\limits_{m=1}^\infty \mathrm{Dom}(\mathcal{L}^m),$$
where $\mathrm{Dom}(\mathcal{L}^m)$ is the domain of the operator $\mathcal{L}^m$, in turn defined as
$$\mathrm{Dom}(\mathcal{L}^m):=\left\{f\in L^2(0,1): \mathcal{L}^j f\in \mathrm{Dom}(\mathcal{L}),\,\, j=0,1,2,...,m-1\right\}.$$
The Fréchet topology of $C^\infty_\mathcal{L}(0,1)$ is given by the family of norms 
\begin{equation}\label{frechet}
    \|\phi\|_{C^m_\mathcal{L}}:=\max\limits_{j\leq m}\|\mathcal{L}^j\phi\|_{L^2(0,1)},\quad m\in \mathbb{N}_0,\,\, \phi\in C^\infty_\mathcal{L}(0,1).
\end{equation}
The space of $\mathcal{L}$-distributions
$$\mathcal{D}'_\mathcal{L}(0,1):=\mathbf{L}\left(C^\infty_\mathcal{L}(0,1),\mathbb{C}\right)$$
is the space of all linear continuous functionals on $C^\infty_\mathcal{L}(0,1)$. For $\omega \in \mathcal{D}'_\mathcal{L}(0,1)$ and $\phi\in C^\infty_\mathcal{L}(0,1)$, we shall write 
$$\omega(\phi)=\langle \omega, \phi\rangle.$$
For any $\psi \in C^\infty_\mathcal{L}(0,1)$, the functional 
$$C^\infty_\mathcal{L}(0,1)\ni \phi \mapsto \int\limits_0^1 \psi(x)\phi(x)dx$$
is an $\mathcal{L}$-distribution, which gives an embedding $\psi \in C^\infty_\mathcal{L}(0,1)\hookrightarrow \mathcal{D}'_\mathcal{L}(0,1)$.

We introduce the spaces $C^j([0,T],W^k_\mathcal{L}(0,1))$ given by the family of norms
\begin{equation}
    \|f\|_{C^n([0,T],W^k_\mathcal{L}(0,1))}=\max\limits_{0\leq t\leq T}\sum\limits_{j=0}^n\left\|\partial^j_t f(t,\cdot)\right\|_{W^k_\mathcal{L}},
\end{equation}
where $k\in \mathbb{R}, \, f\in C^j([0,T],W^k_\mathcal{L}(0,1)).$
 
\begin{thm}\label{th1}
Assume that $q \in L^\infty(0,1)$, $a\in L^\infty[0,T]$, and $a(t)\geq a_0>0$ for all $t\in [0,T]$. For any $k\in \mathbb{R}$, if the initial condition satisfies $u_0 \in W^k_\mathcal{L}(0,1)$ then the heat equation \eqref{C.p1} with the initial/boundary problem  \eqref{C.p2}-\eqref{C.p3}  has a unique solution $u\in C([0,T], W^{k}_\mathcal{L}(0,1)$. We also have the following estimates:
\begin{equation}\label{eq2.1}
    \|u(t,\cdot)\|_{L^2}\lesssim \|u_0\|_{L^2},
\end{equation}
\begin{equation}\label{eq2.2}
    \|\partial_{t}u(t,\cdot)\|_{L^2}\lesssim\|a\|_{L^\infty[0,T]}\|u_0\|_{W^2_\mathcal{L}},
\end{equation}
\begin{equation}\label{est3}
    \|\partial_{x}u(t,\cdot)\|_{L^2}\lesssim\|u_0\|_{W^1_\mathcal{L}}\left(1+\|\nu\|_{L^2}\right)+\|u_0\|_{L^2}\|\nu\|_{L^\infty},
\end{equation}
\begin{equation}\label{est2.4}
\left\|\partial_xu(t,\cdot)\right\|_{L^2} 
\lesssim \|q\|_{L^\infty} \|u_0\|_{L^2}+\|u_0\|_{W^2_\mathcal{L}},
\end{equation}
\begin{equation}\label{est5}
 \left\|u(t, \cdot)\right\|_{W^k_\mathcal{L}}\lesssim\left\|u_0\right\|_{W^k_\mathcal{L}},
\end{equation}
where the constants in these inequalities are independent of $u_0$, $\nu$, $q$ and $a$.
\end{thm} 

We note that $q\in L^\infty(0,1)$ implies that $\nu \in L^\infty(0,1)$ and hence $\nu\in L^2(0,1)$, so that the formulas in the introduction hold true.

\begin{proof}
Using the method of separating variables (see, for example, \cite{Separ}), we will look for a solution to the problem \eqref{C.p1}-\eqref{C.p3} in the form
$$u(t,x)=T(t)X(x),$$
where $T(t)$, $X(x)$ are unknown functions that must be determined. Substituting $u(t,x)=T(t)X(x)$ into the equation \eqref{C.p1}, we obtain the equation
$$T'(t)X(x)+a(t)\left(-T(t)X''(x)+q(x)T(t)X(x)\right)=0.$$
Dividing by $a(t)T(t)X(x)$, we have
\begin{equation}\label{3-1}
    \frac{T'(t)}{a(t)T(t)}=\frac{X''(x)-q(x)X(x)}{X(x)}=-\lambda,
\end{equation}
for some constant $\lambda$. Therefore, if there exists a solution $u(t,x) = T(t)X(x)$ of the heat equation, then $T(t)$ and $X(x)$ must satisfy the equations
$$\frac{T'(t)}{a(t)T(t)}=-\lambda,$$
$$\frac{X''(x)-q(x)X(x)}{X(x)}=-\lambda,$$
for some constant $\lambda$. In addition, in order for $u$ to satisfy the boundary conditions \eqref{C.p3}, we need our function $X$ to satisfy the boundary conditions \eqref{Dirihle}. That is, we need to find a function $X$ and a scalar $\lambda$, such that
\begin{equation}\label{4}
    -X''(x)+q(x)X(x)=\lambda X(x),
\end{equation}
\begin{equation}\label{5}
    X(0)=X(1)=0.
\end{equation}

Thus, we have obtained an eigenvalue problem for the Sturm-Liouville operator. That is, the equation \eqref{4} with the boundary condition \eqref{5} has eigenvalues of the form \eqref{e-val} with corresponding eigenfunctions of the form \eqref{sol-SL}.

Solving the left side of the equation \eqref{3-1} with respect to the independent variable $t$, i.e.
\begin{equation}\label{3}
        T'(t)=-\lambda T(t)a(t),  \qquad t\in [0,T],
\end{equation}
with the initial conditions \eqref{C.p2}, we get
$$T_n(t)=B_n e^{-\lambda_n\int\limits_0^ta(\tau)d\tau},$$
where
$$B_n=\int\limits_0^1u_0(x)\phi_n(x)dx,$$
since by \eqref{sol-SL} $\varphi_n$ are all real.

Consequently, the solution of the heat equation \eqref{C.p1}  with the initial/boundary conditions \eqref{C.p2}-\eqref{C.p3} has the form
\begin{equation}\label{23}
    u(t,x)=\sum\limits_{n=1}^\infty B_ne^{-\lambda_n\int\limits_0^ta(\tau)d\tau}\phi_n(x). 
\end{equation}

Let us prove that $u\in C^2([0,T],L^2(0,1))$. By using the Cauchy-Schwarz inequality for a fixed $t$, we get

\begin{eqnarray}\label{25-0}
\|u(t, \cdot)\|^2_{L^2}&=&\int\limits_0^1|u(t,x)|^2dx =\int\limits_0^1\left|\sum\limits_{n=1}^\infty B_n e^{-\lambda_n\int\limits_0^ta(\tau)d\tau}\phi_n(x)\right|^2dx\nonumber\\
&\lesssim& \int\limits_0^1\sum\limits_{n=1}^\infty\left|B_n e^{-\lambda_n\int\limits_0^ta(\tau)d\tau}\right|^2|\phi_n(x)|^2dx.
\end{eqnarray}
Since the first parenthesis in the last expression does not depend on $x$, we have
$$\int\limits_0^1\sum\limits_{n=1}^\infty\left|B_n e^{-\lambda_n\int\limits_0^ta(\tau)d\tau}\right|^2|\phi_n(x)|^2dx=\sum\limits_{n=1}^\infty\left|B_n e^{-\lambda_n\int\limits_0^ta(\tau)d\tau}\right|^2\int\limits_0^1|\phi_n(x)|^2dx,$$
and according to \eqref{norm-phi}
\begin{equation}\label{phi=1}
  \int\limits_0^1|\phi_n(x)|^2dx=1.  
\end{equation}
Taking into account \eqref{e-val}, we obtain 
\begin{eqnarray}\label{B_n}
    \left|B_n e^{-\lambda_n\int\limits_0^ta(\tau)d\tau}\right|^2\leq |B_n|^2\left|e^{-\int\limits_0^ta(\tau)d\tau}\right|^2=|B_n|^2|g(t)|^2,
\end{eqnarray}
Now, let us estimate $|g(t)|^2$, where
$$g(t)=\exp{\left\{-\int\limits_0^ta(\tau)d\tau\right\}}.$$
Since $0<a_0\leq a(t)$ at $t\in [0,T]$, we have $$|g(t)|^2\leq \|g\|^2_{L^\infty}=1,$$ 
therefore,
\begin{eqnarray}\label{25}
\|u(t, \cdot)\|^2_{L^2}&\lesssim&\sum\limits_{n=1}^\infty\left|B_n e^{-\lambda_n\int\limits_0^ta(\tau)d\tau}\right|^2\leq \sum\limits_{n=1}^\infty\left|B_n\right|^2.
\end{eqnarray}
By using the Parseval identity,
we get
\begin{equation}\label{B_n+}
 \sum\limits_{n=1}^\infty |B_n|^2=\int\limits_0^1|u_0(x)|^2dx=\|u_0\|^2_{L^2}.
\end{equation}

Therefore, we obtain
$$
\|u(t,\cdot)\|^2_{L^2}\lesssim \|u_0\|^2_{L^2}.
$$

Now, let us estimate $\|\partial_t u(t,\cdot)\|^2_{L^2}$ by using \eqref{phi=1} and \eqref{B_n},
\begin{eqnarray}\label{t26}
\|\partial_t u(t,\cdot)\|^2_{L^2}&=&\int\limits_0^1|\partial_tu(t,x)|^2dt = \int\limits_0^1\left|\sum\limits_{n=1}^\infty\left(-\lambda_n a(t)B_n e^{-\lambda_n \int\limits_0^ta(\tau)d\tau}\phi_n(x)\right)\right|^2dx \nonumber\\ 
&\lesssim& \int\limits_0^1\sum\limits_{n=1}^\infty\left|\lambda_n a(t)B_n e^{-\lambda_n\int\limits_0^ta(\tau)d\tau}\right|^2|\phi_n(x)|^2dx\nonumber\\
&\leq&  \sum\limits_{n=1}^\infty |a(t)|^2|\lambda_n B_n |^2\leq\|a\|^2_{L^\infty[0,T]}\sum\limits_{n=1}^\infty|\lambda_n B_n |^2.
\end{eqnarray}
Since $\lambda_n$ are eigenvalues with corresponding eigenfunctions $\phi_n(x)$ of the operator $\mathcal{L}$, we obtain 
\begin{eqnarray}\label{21-1}
\sum\limits_{n=1}^\infty|\lambda_n B_n|^2&=& \sum\limits_{n=1}^\infty\left|\lambda_n\int\limits_0^1 u_0(x)\phi_n(x)dx\right|^2 =\sum\limits_{n=1}^\infty\left| \int\limits_0^1 \lambda_nu_0(x)\phi_n(x)dx\right|^2\nonumber\\
&=&\sum\limits_{n=1}^\infty\left| \int\limits_0^1 \mathcal{L}u_0(x)\phi_n(x)dx\right|^2.
      \end{eqnarray}
From Parseval's identity it follows that
\begin{equation}\label{21-2}
\sum\limits_{n=1}^\infty\left| \int\limits_0^1 \mathcal{L}u_0(x)\phi_n(x)dx\right|^2=\|\mathcal{L}u_0\|^2_{L^2}=\|u_0\|^2_{W^2_\mathcal{L}}.
\end{equation}
Thus, 
$$\|\partial_t u(t,\cdot)\|^2_{L^2}\lesssim \|a\|^2_{L^\infty[0,T]}\|u_0\|^2_{W^2_\mathcal{L}}.$$

Further on, we will estimate $\|\partial_x u(t,\cdot)\|^2_{L^2}$ by using \eqref{phi-der} and \eqref{norm-phi} for $\phi'_n$:
\begin{eqnarray*}
\|\partial_x u(t,\cdot)\|^2_{L^2}&=&\int\limits_0^1|\partial_xu(t,x)|^2dt=
\int\limits_0^1\left|\sum\limits_{n=1}^\infty B_ne^{-\lambda_n\int\limits_0^ta(\tau)d\tau}\phi'_n(x)\right|^2dx \nonumber \\ &=&\int\limits_0^1\left|\sum\limits_{n=1}^\infty B_ne^{-\lambda_n\int\limits_0^ta(\tau)d\tau}\left(\frac{\sqrt{\lambda_n}r_n(x) \cos\theta_n(x)}{\|\Tilde{\phi}_n\|_{L^2}}+\nu(x)\phi_n(x)\right)\right|^2dx.
\end{eqnarray*}
In compliance with \eqref{25}, \eqref{est-high} and \eqref{est_low}, there exist some $C_0>0$, such that $C_0<\|\Tilde{\phi}_n\|_{L^2}<\infty$, so that
\begin{eqnarray*}
\|\partial_x u(t,\cdot)\|^2_{L^2}&\lesssim&\sum\limits_{n=1}^\infty\left|\sqrt{\lambda_n}B_n\right|^2\int\limits_0^1|r_n(x)|^2dx\\
&+&\sum\limits_{n=1}^\infty|B_n|^2\int\limits_0^1|\nu(x)\phi_n(x)|^2dx.
\end{eqnarray*}
For $r_n(x)$ according Theorem 2 in \cite{Savch}, we have
$$r_n(x)=1+\rho_n(x),\quad \|\rho_n\|_{L^2}\lesssim \|\nu\|_{L^2},$$
where the constant is independent of $\nu$ and $n$. As a result,
$$\int\limits_0^1|r_n(x)|^2dx\lesssim 1+\|\nu\|^2_{L^2}.$$
For the second term we obtain
$$\int\limits_0^1|\nu(x)\phi_n(x)|^2dx\leq \|\nu\|^2_{L^\infty}\|\phi_n\|^2_{L^2}=\|\nu\|^2_{L^\infty},$$
since $\varphi_n$ are orthonormal in $L^2(0,1)$.
Using the property of the operator $\mathcal{L}$ in the Sobolev space $W_{\mathcal{L}}^k$ and the Parseval identity, we obtain
\begin{eqnarray*}
 \sum\limits_{n=1}^\infty\left|\sqrt{\lambda_n}B_n\right|^2&=& \sum\limits_{n=1}^\infty\left|\int\limits_0^1\sqrt{\lambda_n}u_0(x)\phi_n(x)dx\right|^2\\
 &=&\sum\limits_{n=1}^\infty\left|\int\limits_0^1\mathcal{L}^\frac{1}{2}u_0(x)\phi_n(x)dx\right|^2=\left\|\mathcal{L}^\frac{1}{2}u_0\right\|^2_{L^2}=\|u_0\|^2_{W^1_\mathcal{L}}.
\end{eqnarray*}
Taking into account the last relations, we obtain
\begin{eqnarray}\label{u_x}
\|\partial_x u(t,\cdot)\|^2_{L^2}&\lesssim&\sum\limits_{n=1}^\infty\left|\sqrt{\lambda_n}B_n\right|^2\left(1+\|\nu\|^2_{L^2}\right)+\sum\limits_{n=1}^\infty \left|B_n\right|^2\|\nu\|^2_{L^\infty}\nonumber\\
&\leq& \|u_0\|^2_{W^1_\mathcal{L}}\left(1+\|\nu\|^2_{L^2}\right)+\|u_0\|^2_{L^2}\|\nu\|^2_{L^\infty},
\end{eqnarray}
which implies \eqref{est3}.

For the estimate of $\left\|\partial_x^2u(t, \cdot)\right\|_{L^2}$ we use $\phi''_n(x)=(q(x)-\lambda_n)\phi_n(x)$ and \eqref{B_n}, so that
\begin{eqnarray}\label{u_xx}
\left\|\partial_x^2u(t, \cdot)\right\|^2_{L^2}&=&\int\limits_0^1\left|\partial^2_xu(t,x)\right|^2dx=\int\limits_0^1\left|\sum\limits_{n=1}^\infty B_n e^{-\lambda_n\int\limits_0^ta(\tau)d\tau} \phi''_n(x)\right|^2dx\nonumber\\
&\lesssim& \int\limits_0^1\left(\sum\limits_{n=1}^\infty \left|{B_n} e^{-\lambda_n\int\limits_0^ta(\tau)d\tau}\right|^2|(q(x)-\lambda_n)\phi_n(x)|^2\right)dx\nonumber\\
&\lesssim&\int\limits_0^1|q(x)|^2\sum\limits_{n=1}^\infty \left|B_n\right|^2|\phi_n(x)|^2dx+\int\limits_0^1\sum\limits_{n=1}^\infty |\lambda_n B_n|^2|\phi_n(x)|^2 dx\nonumber\\
&\leq&\|q\|^2_{L^\infty}\sum\limits_{n=1}^\infty |B_n|^2 +\sum\limits_{n=1}^\infty \left|\lambda_n B_n\right|^2. 
\end{eqnarray}
Using \eqref{B_n+}, \eqref{21-1} and \eqref{21-2} for \eqref{u_xx}, we obtain
\begin{eqnarray*}
\left\|\partial^2_xu(t,\cdot)\right\|^2_{L^2} 
&\lesssim &\|q\|^2_{L^\infty} \|u_0\|^2_{L^2}+\|u_0\|^2_{W^2_\mathcal{L}},
\end{eqnarray*}
implying \eqref{est2.4}.

Finally, we will get the last estimate \eqref{est5} using that $\mathcal{L}^ku=\lambda_n^ku$ and Parseval's identity:
\begin{eqnarray*}
 \left\|u(t, \cdot)\right\|^2_{W^k_\mathcal{L}}&=&\left\|\mathcal{L}^\frac{k}{2}u(t, \cdot)\right\|^2_{L^2}=\int\limits_0^1\left|\mathcal{L}^\frac{k}{2}u(t,x)\right|^2dx\\
&=&\int\limits_0^1\left|\sum\limits_{n=1}^\infty  B_n e^{-\lambda_n\int\limits_0^ta(\tau)d\tau}\lambda_n^\frac{k}{2}\phi_n(x)\right|^2dx\lesssim\sum\limits_{n=1}^\infty \left| \lambda_n^\frac{k}{2}B_n\right|^2\\
&=&\left\|\mathcal{L}^\frac{k}{2}u_0\right\|^2_{L^2}=\left\|u_0\right\|^2_{W^k_\mathcal{L}}.
\end{eqnarray*}
The proof of Theorem \ref{th1} is complete.
\end{proof}

We establish the following statement to eliminate the dependence of Sobolev spaces on the operator $\mathcal{L}$, while the estimates will depend only on the initial data $u_0$ and potential $q$. This will be important for further analysis.

\begin{cor}\label{cor1}
Assume that $q,\, \nu \in L^\infty(0,1)$, $a\in L^\infty[0,T]$, and $a(t)\geq a_0>0$ for all $t\in [0,T]$. If the initial condition satisfies $u_0 \in L^2(0,1)$ and $u_0''\in L^2(0,1)$ then the heat equation \eqref{C.p1} with the initial/boundary conditions  \eqref{C.p2}-\eqref{C.p3}  has a unique solution $u\in C([0,T], L^2(0,1))$ which satisfies the estimates
\begin{equation}\label{ec1}
      \|u(t,\cdot)\|_{L^2}\lesssim \|u_0\|_{L^2},
\end{equation}
\begin{equation}\label{ec2}
\|\partial_t u(t,\cdot)\|_{L^2}\lesssim \|a\|_{L^\infty[0,T]}\left(\|u''_0\|_{L^2}+\|q\|_{L^\infty}\|u_0\|_{L^2}\right),
\end{equation}
\begin{eqnarray}\label{ec3}
\|\partial_{x}u(t,\cdot)\|_{L^2}\lesssim \left(\|u''_0\|_{L^2}+\|q\|_{L^\infty}\|u_0\|_{L^2}\right)\left(1 +\|\nu\|_{L^2}\right)+  \|u_0\|_{L^2}\|\nu\|_{L^\infty},
\end{eqnarray}
\begin{equation}\label{ec4}
\left\|\partial^2_xu(t,\cdot)\right\|_{L^2} 
\lesssim \|u''_0\|_{L^2}+\|q\|_{L^\infty}\|u_0\|_{L^2},
\end{equation}
where  the constants in these inequalities are independent of $u_0$, $\nu$, $q$ and $a$.
\end{cor}
\begin{proof}
The estimate \eqref{ec1} immediately follows from Theorem \ref{th1} so we move on to the next estimates.

By \eqref{t26} we have
\begin{equation*}
\|\partial_t u(t,\cdot)\|^2\lesssim \|a\|^2_{L^\infty[0,T]}\sum\limits_{n=1}^\infty|\lambda_n B_n |^2.
\end{equation*}
Since $\lambda_n$ are the eigenvalues of the operator $\mathcal{L}$, we obtain 
\begin{eqnarray}\label{21}\sum\limits_{n=1}^\infty|\lambda_n B_n|^2&=& \sum\limits_{n=1}^\infty\left| \int\limits_0^1 \lambda_n u_0(x)\phi_n(x)dx\right|^2
=\sum\limits_{n=1}^\infty\left| \int\limits_0^1 \left(-u''_0(x)+q(x)u_0(x)\right)\phi_n(x)dx\right|^2\nonumber \\
&\lesssim& \sum\limits_{n=1}^\infty\left| \int\limits_0^1 u''_0(x)\phi_n(x)dx\right|^2+\sum\limits_{n=1}^\infty\left|\int\limits_0^1q(x)u_0(x)\phi_n(x)dx\right|^2.
      \end{eqnarray}
Using Parseval's identity and since $q\in L^\infty$, we have
$$\sum\limits_{n=1}^\infty\left|\int\limits_0^1q(x)u_0(x)\phi_n(x)dx\right|^2=\sum\limits_{n=1}^\infty\left|\langle (q u_0),\phi_n\rangle\right|^2=\|q u_0\|^2_{L^2}\leq \|q\|^2_{L^\infty} \|u_0\|^2_{L^2},$$
and also we use Parseval's identity for the first term
$$\sum\limits_{n=1}^\infty\left| \int\limits_0^1 u''_0(x)\phi_n(x)dx\right|^2=\sum\limits_{n=1}^\infty |u_{0,n}''|^2=\|u_0''\|^2_{L^2}.$$
Consequently
\begin{eqnarray}\label{LA}\sum\limits_{n=1}^\infty|\lambda_nB_n|^2&\lesssim& \|u_0''\|^2_{L^2}+\|q\|^2_{L^\infty}\|u_0\|^2_{L^2}.
\end{eqnarray}
Thus, 
$$\|\partial_t u(t,\cdot)\|^2_{L^2}\lesssim \|a\|^2_{L^\infty[0,T]}\left(\|u''_0\|^2_{L^2}+\|q\|^2_{L^\infty}\|u_0\|^2_{L^2}\right),$$
proving \eqref{ec2}.

From \eqref{u_x} we have
\begin{eqnarray*}
\|\partial_x u(t,\cdot)\|^2_{L^2}&=&\int\limits_0^1|\partial_xu(t,x)|^2dt
=
\int\limits_0^1\left|\sum\limits_{n=1}^\infty B_n e^{-\lambda_n\int\limits_0^ta(\tau)d\tau}\phi'_n(x)\right|^2dx\\
&\lesssim&\sum\limits_{n=1}^\infty\left|\sqrt{\lambda_n}B_n\right|^2\left(1+\|\nu\|^2_{L^2}\right)+\sum\limits_{n=1}^\infty \left|B_n\right|^2\|\nu\|^2_{L^\infty}.
\end{eqnarray*}
Taking into account \eqref{e-val}, using \eqref{LA} and Parseval's identity we get
\begin{eqnarray*}
\|\partial_x u(t,\cdot)\|^2_{L^2}&\lesssim&\sum\limits_{n=1}^\infty\left|\sqrt{\lambda_n}B_n\right|^2\left(1+\|\nu\|^2_{L^2}\right)+\sum\limits_{n=1}^\infty \left|B_n\right|^2\|\nu\|^2_{L^\infty}\\
&\leq& \sum\limits_{n=1}^\infty\left|\lambda_n B_n\right|^2\left(1+\|\nu\|^2_{L^2}\right)+\sum\limits_{n=1}^\infty \left|B_n\right|^2\|\nu\|^2_{L^\infty}\\
&\lesssim& \left(\|u''_0\|^2_{L^2}+\|q\|^2_{L^\infty}\|u_0\|^2_{L^2}\right)\left(1 +\|\nu\|^2_{L^2}\right)+  \|u_0\|^2_{L^2}\|\nu\|^2_{L^\infty},
\end{eqnarray*}
implying \eqref{ec3}. 

For $\left\|\partial_x^2u(t, \cdot)\right\|^2_{L^2}$ from \eqref{u_xx} we have
\begin{eqnarray*}
\left\|\partial_x^2u(t, \cdot)\right\|^2_{L^2}&=&\int\limits_0^1\left|\partial^2_xu(t,x)\right|^2dx=\int\limits_0^1\left|\sum\limits_{n=1}^\infty B_n e^{-\lambda_n\int\limits_0^ta(\tau)d\tau} \phi''_n(x)\right|^2dx\nonumber\\
&\lesssim& \|q\|^2_{L^\infty}\sum\limits_{n=1}^\infty |B_n|^2 +\sum\limits_{n=1}^\infty \left|\lambda_n B_n\right|^2. \end{eqnarray*}
Using \eqref{LA} and Parseval’s identity  we obtain
\begin{eqnarray*}
\left\|\partial_x^2u(t, \cdot)\right\|^2_{L^2}&\lesssim&\|q\|^2_{L^\infty}\sum\limits_{n=1}^\infty |B_n|^2 +\sum\limits_{n=1}^\infty \left|\lambda_n B_n\right|^2\\
&\lesssim& \|q\|^2_{L^\infty}\|u_0\|^2_{L^2}+\|u''_0\|^2_{L^2}+\|q\|^2_{L^\infty}\|u_0\|^2_{L^2}=\|u''_0\|^2_{L^2}+2\|q\|^2_{L^\infty}\|u_0\|^2_{L^2}.
\end{eqnarray*}
The proof of Corollary \ref{cor1} is complete.
\end{proof}

\section{Non-homogeneous heat equation}

In this section, we study and establish estimates for the non-homogeneous heat equation with initial/boundary conditions 
\begin{equation}\label{nonh}
    \left\{\begin{array}{l}
    \partial_t u(t,x)+a(t)\mathcal{L} u(t,x)=f(t,x),\qquad (t,x)\in [0,T]\times (0,1),\\
    u(0,x)=u_0(x),\quad x\in (0,1),\\
    u(t,0)=0=u(t,1),\quad t\in[0,T],
    \end{array}\right.
\end{equation}
where operator $\mathcal{L}$ is defined by 
$$\mathcal{L}=-\frac{\partial^2}{\partial x^2}+q(x),\qquad x\in(0,1).$$

\begin{thm}\label{non-hom}
Assume that $q \in L^\infty(0,1)$, $a(t)\geq a_0>0$ for all $t\in [0,T]$, $a'\in L^\infty[0,T]$ and $f\in C^1([0,T],W^k_\mathcal{L}(0,1))$ for some $k\in \mathbb{R}$. If the initial condition satisfies $u_0 \in W^{k}_\mathcal{L}$ then the non-homogeneous heat equation with initial/boundary conditions \eqref{nonh} has the unique solution $u\in C([0,T], W^{k}_\mathcal{L})$ which then satisfies the estimates
\begin{equation}\label{es-nh1}
        \|u(t,\cdot)\|_{L^2}\lesssim \|u_0\|_{L^2}+T\|f\|_{C([0,T],L^2(0,1))},
\end{equation}
\begin{equation}\label{es-nh2}
\|\partial_tu(t,\cdot)\|_{L^2}\lesssim \|a\|_{L^\infty[0,T]}\left(\|u_0\|_{W^1_\mathcal{L}}+\frac{T}{a_0} \|f\|_{C^1([0,T],L^2(0,1))}\right),
\end{equation}
\begin{eqnarray}\label{es-nh3}
\|\partial_xu(t,\cdot)\|_{L^2}&\lesssim& \left(1+\|\nu\|_{L^2}\right)\left(\|u_0\|_{W^1_\mathcal{L}}+T\|f\|_{C([0,T],W^1_\mathcal{L}(0,1))}\right),
\end{eqnarray}
\begin{eqnarray}\label{es-nh4}
\|\partial^2_xu(t,\cdot)\|_{L^2}&\lesssim& \|q\|_{L^\infty} \left(\|u_0\|_{L^2}+T\|f\|_{C([0,T],L^2(0,1))}\right)\nonumber\\
&+&\|u_0\|_{W^2_\mathcal{L}}+\frac{T}{a_0}\|f\|_{C^1([0,T],L^2(0,1))},
\end{eqnarray}
where the constants in these inequalities are independent of $u_0$, $q$, $a$ and $f$.
\end{thm}

\begin{proof}
We will employ the eigenfunctions described by equation \eqref{sol-SL} for the related eigenvalue problem \eqref{4}-\eqref{5} and seek a solution in the form of a series

\begin{equation}\label{nonhu}
    u(t,x)=\sum\limits_{n=1}^\infty u_n(t)\phi_n(x),
\end{equation}
where
$$u_n(t) =\int\limits_0^1 u(t,x)\phi_n(x)dx.$$
Likewise, we can also express the source function through an expansion,
\begin{equation}\label{func}
    f(t,x)=\sum\limits_{n=1}^\infty f_n(t)\phi_n(x),\qquad f_n(t)=\int\limits_0^1f(t,x)\phi_n(x)dx.
\end{equation}

Now, since we are looking for a function $u(t, x)$ that is twice differentiable and satisfies the homogeneous Dirichlet boundary conditions, we can differentiate the Fourier series \eqref{nonhu} elementwise. Using the fact that $\phi_n$ satisfies equation \eqref{4}, we can then derive the following:
\begin{equation}\label{uxx}
    u_{xx}(t,x) = \sum\limits_{n=1}^\infty u_n(t)\phi''_n(x)=\sum\limits_{n=1}^\infty u_n(t)(q(x)-\lambda_n)\phi_n(x).
\end{equation}

Similarly, we can differentiate the series \eqref{func} with respect to the variable $t$ to obtain
\begin{equation}\label{utt}
    u_{t}(t,x) = \sum\limits_{n=1}^\infty u'_n(t)\phi_n(x),
\end{equation}
considering that the Fourier coefficients of $u_{t}(t,x)$ are as follows:
$$\int\limits_0^1u_{t}(t,x)\phi_n(x)dx=\frac{\partial}{\partial t}\left[\int\limits_0^1u(t,x)\phi_n(x)dx\right]=u'_n(t).$$
Differentiation under the above integral is allowed because the resulting integrand remains continuous.

By substituting \eqref{utt} and \eqref{uxx} into equation \eqref{nonh} and making use of \eqref{func}, we arrive at the following result:
\begin{eqnarray*}
\sum\limits_{n=1}^\infty u'_n(t)\phi_n(x)&+&a(t)\left(-\sum\limits_{n=1}^\infty u_n(t)\left(q(x)-\lambda_n\right)\phi_n(x)+q(x)\sum\limits_{n=1}^\infty u_n(t)\phi_n(x)\right)\\
&=&\sum\limits_{n=1}^\infty f_n(t)\phi_n(x),
\end{eqnarray*}
after a minor rearrangement, we obtain
$$\sum\limits_{n=1}^\infty \left[u'_n(t)+\lambda_na(t)u_n(t)\right]\phi_n(x)=\sum\limits_{n=1}^\infty f_n(t)\phi_n(x).$$
But then, by virtue of the completeness property,
$$u'_n(t)+\lambda_na(t)u_n(t)=f_n(t), \qquad n=1,2,...,$$
these are ordinary differential equations (ODEs) for the coefficients $u_n(t)$ in the series \eqref{nonhu}. Applying the method of variation of constants, we obtain
\begin{eqnarray*}
u_n(t)&=&B_n e^{-\lambda_n\int\limits_0^ta(\tau)d\tau}+ e^{-\lambda_n\int\limits_0^ta(\tau)d\tau} \int\limits_0^t e^{\lambda_n \int\limits_0^s a(\tau)d\tau} f_n(s)ds,
\end{eqnarray*}
where
$$B_n=\int\limits_0^1u_0(x)\phi_n(x)dx.$$
Therefore, we can express a solution of the equation \eqref{nonh} in the following form:
\begin{eqnarray}\label{sol-nh}
u(t,x)&=&\sum\limits_{n=0}^\infty B_n e^{-\lambda_n\int\limits_0^ta(\tau)d\tau}\phi_n(x) \nonumber \\
&+&\sum\limits_{n=0}^\infty  e^{-\lambda_n\int\limits_0^ta(\tau)d\tau} \int\limits_0^t e^{\lambda_n\int\limits_0^sa(\tau)d\tau}f_n(s)ds\phi_n(x).
\end{eqnarray}

Further on, we will proceed to estimate  $\|u(t,\cdot)\|^2_{L^2}$. For this we use the expressions
\begin{eqnarray}\label{est-nonh}
\int\limits_0^1|u(t,x)|^2dx&=&\int\limits_0^1\left|\sum\limits_{n=0}^\infty B_n e^{-\lambda_n\int\limits_0^ta(\tau)d\tau}\phi_n(x)\right.\nonumber\\
&+&\left.\sum\limits_{n=0}^\infty  e^{-\lambda_n\int\limits_0^ta(\tau)d\tau} \int\limits_0^t e^{\lambda_n\int\limits_0^sa(\tau)d\tau}f_n(s)ds\phi_n(x)\right|^2 dx\nonumber \\
&\lesssim& \int\limits_0^1\left|\sum\limits_{n=0}^\infty B_n e^{-\lambda_n\int\limits_0^ta(\tau)d\tau}\phi_n(x)\right|^2dx\nonumber\\
&+&\int\limits_0^1\left|\sum\limits_{n=0}^\infty  e^{-\lambda_n\int\limits_0^ta(\tau)d\tau} \int\limits_0^t e^{\lambda_n\int\limits_0^sa(\tau)d\tau}f_n(s)ds\phi_n(x)\right|^2 dx\nonumber \\
&=& I_1 + I_2.
\end{eqnarray}
For $I_1$, we can use \eqref{25-0} -\eqref{B_n+} in the homogeneous case, leading to the following result
$$I_1:=\int\limits_0^1\left|\sum\limits_{n=0}^\infty B_n e^{-\lambda_n\int\limits_0^ta(\tau)d\tau}\phi_n(x)\right|^2dx\lesssim \|u_0\|^2_{L^2}.$$
Taking into account that $s\in [0,t]$ and $a(s)>0$ for $I_2$ in \eqref{est-nonh} we get
\begin{eqnarray}\label{II2}
I_2&:=&\int\limits_0^1\left|\sum\limits_{n=0}^\infty  e^{-\lambda_n\int\limits_0^ta(\tau)d\tau} \int\limits_0^t e^{\lambda_n\int\limits_0^sa(\tau)d\tau}f_n(s)ds\phi_n(x)\right|^2 dx\nonumber\\
&\leq& \int\limits_0^1\left|\sum\limits_{n=0}^\infty  e^{-\lambda_n\int\limits_0^ta(\tau)d\tau} e^{\lambda_n\int\limits_0^ta(\tau)d\tau} \int\limits_0^t f_n(s)ds\phi_n(x)\right|^2 dx\nonumber\\
&=&\int\limits_0^1\left|\sum\limits_{n=0}^\infty \int\limits_0^t f_n(s)ds\phi_n(x)\right|^2 dx\lesssim \sum\limits_{n=1}^\infty \left[\int\limits_0^t|f_n(s)|ds\right]^2.
\end{eqnarray}
Applying Hölder's inequality and taking into account that $t\in [0, T]$, we obtain
$$\left[\int\limits_0^t|f_n(s)|ds\right]^2\leq \left[\int\limits_0^T 1\cdot| f_n(t)|dt\right]^2\leq T\int\limits_0^T| f_n(t)|^2dt,$$
since $f_n(t)$ is the Fourier coefficient of the function $f(t,x)$, and by Parseval's identity we get
$$
\sum\limits_{n=1}^\infty T\int\limits_0^T|f_n(t)|^2dt= T\int\limits_0^T\sum\limits_{n=1}^\infty|f_n(t)|^2dt = T\int\limits_0^T\|f(t,\cdot)\|^2_{L^2}dt.$$
Since
$$\|f\|_{C([0,T],L^2(0,1))}=\max\limits_{0\leq t\leq T}\|f(t,\cdot)\|_{L^2},$$
we can deduce the following inequality
$$T\int\limits_0^T\|f(t,\cdot)\|^2_{L^2}dt\leq T^2\|f\|^2_{C([0,T],L^2(0,1))}.$$
Thus,
\begin{eqnarray}\label{I2}
I_2&=&\int\limits_0^1\left|\sum\limits_{n=0}^\infty  e^{-\lambda_n\int\limits_0^ta(\tau)d\tau} \int\limits_0^t e^{\lambda_n\int\limits_0^sa(\tau)d\tau}f_n(s)ds\phi_n(x)\right|^2 dx\nonumber\\
&\lesssim&  T^2\|f\|^2_{C([0,T],L^2(0,1))}.
\end{eqnarray}
Finally, for $\|u(t,\cdot)\|^2_{L^2}$ we get
$$
    \|u(t,\cdot)\|^2_{L^2}\lesssim \|u_0\|^2_{L^2}+T^2\|f\|^2_{C([0,T],L^2(0,1))}.
$$
Let us estimate the next norm $\|\partial_tu(t,\cdot)\|_{L^2}$, for this we deduce  $\partial_tu(t,x)$ as follows
\begin{eqnarray*}
\partial_tu(t,x)&=&\sum\limits_{n=0}^\infty \left(-\lambda_na(t)\right) B_n e^{-\lambda_n\int\limits_0^ta(\tau)d\tau}\phi_n(x) \\
&+&\sum\limits_{n=0}^\infty \left(-\lambda_na(t)\right) e^{-\lambda_n\int\limits_0^ta(\tau)d\tau} \int\limits_0^t e^{\lambda_n \int\limits_0^sa(\tau)d\tau}f_n(s)ds\phi_n(x)\\
&+&\sum\limits_{n=0}^\infty f_n(t)\phi_n(x).
\end{eqnarray*}
Then
\begin{eqnarray*}
\|\partial_tu(t,\cdot)\|^2_{L^2}&=&\int\limits_0^1|\partial_tu(t,x)|^2dx
\lesssim \int\limits_0^1\left|\sum\limits_{n=0}^\infty \lambda_n a(t)B_n e^{-\lambda_n\int\limits_0^ta(\tau)d\tau}\phi_n(x)\right|^2dx \\
&+&\int\limits_0^1\left|\sum\limits_{n=0}^\infty \lambda_na(t) e^{-\lambda_n\int\limits_0^ta(\tau)d\tau} \int\limits_0^t e^{\lambda_n\int\limits_0^sa(\tau)d\tau}f_n(s)ds\phi_n(x)\right|^2 dx\\
&+&\int\limits_0^1\left|\sum\limits_{n=0}^\infty f_n(t)\phi_n(x)\right|^2 dx = J_1 + J_2 +J_3.
\end{eqnarray*}
For $J_1$ one can use the \eqref{21-1} and \eqref{21-2} for the homogeneous case and taking into account \eqref{func} for the function $f(t,x)$ in $J_3$, we have
$$\|\partial_tu(t,\cdot)\|^2_{L^2}\lesssim \|a\|^2_{L^\infty[0,T]}\|u_0\|^2_{W^1_\mathcal{L}}+J_2+\|f(t,\cdot)\|^2_{L^2}.$$
To estimate $J_2$, integrating by part and carrying out as in \eqref{II2} we obtain 
\begin{eqnarray*}
J_2&:=&\int\limits_0^1\left|\sum\limits_{n=0}^\infty \lambda_n a(t)e^{-\lambda_n\int\limits_0^ta(\tau)d\tau} \int\limits_0^t e^{\lambda_n\int\limits_0^sa(\tau)d\tau}f_n(s)ds\phi_n(x)\right|^2 dx\nonumber \\
&=&\int\limits_0^1\left|\sum\limits_{n=0}^\infty \left(\lambda_n a(t)e^{-\lambda_n\int\limits_0^ta(\tau)d\tau}\frac{1}{\lambda_n}e^{\lambda_n \int\limits_0^sa(\tau)d\tau}\frac{f_n(s)}{a(s)}\biggl|_0^t\right.\right. \nonumber\\
&-& \left.\left. a(t)e^{-\lambda_n\int\limits_0^ta(\tau)d\tau}\int\limits_0^t e^{\lambda_n\int\limits_0^sa(\tau)d\tau}\left(\frac{f_n(s)}{a(s)}\right)'ds\right)\phi_n(x)\right|^2 dx\\
&\lesssim& \sum\limits_{n=0}^\infty \left|a(t)\frac{f_n(s)}{a(s)}\biggl|_0^t\right|^2 + \sum\limits_{n=1}^\infty\left|a(t)\int\limits_0^t \left(\frac{f'_n(s)}{a(s)}-\frac{f_n(s)a'(s)}{a^2(s)}\right)ds\right|^2=J_{21}+J_{22}.
\end{eqnarray*}
Since $a(x)\geq a_0>0$ in $t\in[0,T]$, we can use the estimate $\left|\frac{1}{a(t)}\right|^2\leq \frac{1}{a_0^2}$ for all $t\in[0,T]$ and using Parseval's identity, we obtain
\begin{eqnarray*}
J_{21}&:=& \sum\limits_{n=0}^\infty \left|a(t) \frac{f_n(s)}{a(s)}\biggl|_0^t\right|^2 \lesssim \|a\|^2_{L^\infty[0,T]}\left(\sum\limits_{n=1}^\infty \left|\frac{f_n(t)}{a(t)}\right|^2+\sum\limits_{n=1}^\infty \left|\frac{f_n(0)}{a(0)}\right|^2\right)\\
&\leq& \frac{1}{ a_0^2}\|a\|^2_{L^\infty[0,T]}\left(\sum\limits_{n=1}^\infty \left|f_n(t)\right|^2+\sum\limits_{n=1}^\infty \left|f_n(0)\right|^2\right)\\
&=&\frac{1}{a_0^2} \|a\|^2_{L^\infty[0,T]}\left(\|f(t,\cdot)\|^2_{L^2}+\|f(0,\cdot)\|^2_{L^2}\right).
\end{eqnarray*}
Carrying out similar reasoning, integrating by parts and using \eqref{I2}, we get
\begin{eqnarray*} J_{22}&:=&\sum\limits_{n=1}^\infty\left|a(t)\int\limits_0^t \left(\frac{f'_n(s)}{a(s)}-\frac{f_n(s)a'(s)}{a^2(s)}\right)ds\right|^2\\
&=&\sum\limits_{n=0}^\infty\left|a(t) \left(\int\limits_0^t\frac{f'_n(s)}{a(s)}ds+\int\limits_0^tf_n(s)d\left(\frac{1}{a(s)}\right)\right)\right|^2\\
&\lesssim&\frac{1}{a_0^2}\|a\|^2_{L^\infty[0,T]}\sum\limits_{n=1}^\infty\left|\int\limits_0^t|f'_n(s)|ds\right|^2
+\|a\|^2_{L^\infty[0,T]}\sum\limits_{n=1}^\infty\left| \left(\frac{f_n(s)}{a(s)}\biggl|_0^t-\int\limits_0^t\frac{f'(s)}{a(s)}ds\right)\right|^2\\
&\lesssim&\frac{T^2}{a_0^2}\|a\|^2_{L^\infty[0,T]}\|f'\|^2_{C([0,T],L^2(0,1))}+\frac{1}{a_0^2}\|a\|^2_{L^\infty[0,T]}\left(\|f(t,\cdot)\|^2_{L^2}
+\|f(0,\cdot)\|^2_{L^2}\right)\\
&+&\frac{T^2}{a_0^2}\|a\|^2_{L^\infty[0,T]}\|f'\|^2_{C([0,T],L^2(0,1))}.
\end{eqnarray*}
And finally, for $J_2$ we have
\begin{eqnarray}\label{f'}
J_2&:=&\int\limits_0^1\left|\sum\limits_{n=0}^\infty \lambda_n e^{-\lambda_n\int\limits_0^ta(\tau)d\tau} \int\limits_0^t e^{\lambda_n\int\limits_0^sa(\tau)d\tau}f_n(s)ds\phi_n(x)\right|^2 dx\nonumber\\
&\lesssim& \frac{2}{a_0^2}\|a\|^2_{L^\infty[0,T]} \left(\|f(t,\cdot)\|^2_{L^2}+\|f(0,\cdot)\|^2_{L^2}\right) +2 \frac{T^2}{a_0^2}\|a\|^2_{L^\infty[0,T]}\|f'\|^2_{C([0,T],L^2(0,1)}\nonumber\\
&\lesssim& \frac{2T^2}{a_0^2}\|a\|^2_{L^\infty[0,T]} \|f\|^2_{C^1([0,T],L^2(0,1))}.
\end{eqnarray}
Therefore,
$$\|\partial_tu(t,\cdot)\|^2_{L^2}\lesssim \|a\|^2_{L^\infty[0,T]}\left(\|u_0\|^2_{W^1_\mathcal{L}}+\frac{2T^2}{a_0^2} \|f\|^2_{C^1([0,T],L^2(0,1))}\right).$$

Let us make the next estimate
\begin{eqnarray*}
\|\partial_xu(t,\cdot)\|^2_{L^2}&=&\int\limits_0^1|\partial_xu(t,x)|^2dx
\lesssim \int\limits_0^1\left|\sum\limits_{n=0}^\infty B_n e^{-\lambda_n\int\limits_0^ta(\tau)d\tau}\phi'_n(x)\right|^2 dx \\
&+&\int\limits_0^1\left|\sum\limits_{n=0}^\infty  e^{-\lambda_n\int\limits_0^ta(\tau)d\tau} \int\limits_0^t e^{\lambda_n\int\limits_0^sa(\tau)d\tau}f_n(s)ds\phi'_n(x)\right|^2dx= K_1+K_2.
\end{eqnarray*}
Using \eqref{est3} in the homogeneous case we obtain
\begin{eqnarray*}
K_1&:=& \int\limits_0^1\left|\sum\limits_{n=0}^\infty B_n e^{-\lambda_n\int\limits_0^ta(\tau)d\tau}\phi'_n(x)\right|^2 dx\lesssim \|u_0\|^2_{W^1_\mathcal{L}}\left(1+\|\nu\|^2_{L^2}\right)+\|u_0\|^2_{L^2}\|\nu\|^2_{L^\infty}.
\end{eqnarray*}
For $K_2$ by using \eqref{I2} and \eqref{est3} we get
\begin{eqnarray*}
K_2&:=& \int\limits_0^1\left|\sum\limits_{n=0}^\infty  e^{-\lambda_n\int\limits_0^ta(\tau)d\tau} \int\limits_0^t e^{\lambda_n \int\limits_0^sa(\tau)d\tau}f_n(s)ds\phi'_n(x)\right|^2dx\\
&\leq&\int\limits_0^1\left|\sum\limits_{n=0}^\infty   \int\limits_0^tf_n(s)ds\left(\frac{\sqrt{\lambda_n}r_n(x) \cos\theta_n(x)}{\|\Tilde{\phi}_n\|_{L^2}}+\nu(x)\phi_n(x)\right)\right|^2dx\\
&\lesssim& \sum\limits_{n=0}^\infty \left|\int\limits_0^t\left|\sqrt{\lambda_n}f_n(s)\right|ds\right|^2\left(1+\|\nu\|^2_{L^2}\right) + \|\nu\|^2_{L^\infty}T^2\|f\|^2_{C([0,T],L^2(0,1)},
\end{eqnarray*}
Considering \eqref{func} and \eqref{I2} we have
\begin{eqnarray*}
\sum\limits_{n=0}^\infty \left|\int\limits_0^t\left|\sqrt{\lambda_n}f_n(s)\right|ds\right|^2&\leq&  \sum\limits_{n=0}^\infty \left|\int\limits_0^T\left|\sqrt{\lambda_n}f_n(t)\right|dt\right|^2 \leq T\int\limits_0^T\sum\limits_{n=0}^\infty \left|\sqrt{\lambda_n}f_n(t)\right|^2dt\\
&=&T\int\limits_0^T\sum\limits_{n=0}^\infty \left|\int\limits_0^1\sqrt{\lambda_n}f(t,x)\phi_n(x)dx\right|^2dt\\
&=&T\int\limits_0^T \left\|\mathcal{L}^{\frac{1}{2}}f(t,\cdot)\right\|^2_{L^2}dt\leq T^2\|f\|^2_{C\left([0,T],W^1_\mathcal{L}(0,1)\right)},
\end{eqnarray*}
and we finally obtain
\begin{eqnarray*}
\|\partial_xu(t,\cdot)\|^2_{L^2}&\lesssim&\left(\|u_0\|^2_{W^1_\mathcal{L}}+T^2\|f\|^2_{C\left([0,T],W^1_{\mathcal{L}}(0,1)\right)} \right)\left(1+\|\nu\|^2_{L^2}\right)\\
&+&\left(\|u_0\|^2_{L^2}+T^2\|f\|^2_{C([0,T],L^2(0,1))}\right)\|\nu\|^2_{L^\infty},
\end{eqnarray*}
which gives \eqref{es-nh3}.

For the last estimate $\|\partial^2_xu(t,\cdot)\|^2_{L^2}$ taking into account $\phi_n''(x)=(q(x)-\lambda_n)\phi_n(x)$, we deduce
\begin{eqnarray*}
\|\partial^2_xu(t,\cdot)\|^2_{L^2}&=&\int\limits_0^1|\partial^2_xu(t,x)|^2dx\\
&\lesssim& \int\limits_0^1\left|\sum\limits_{n=0}^\infty B_n e^{-\lambda_n\int\limits_0^ta(\tau)d\tau}(q(x)-\lambda_n)\phi_n(x)\right|^2dx \\
&+&\int\limits_0^1\left|\sum\limits_{n=0}^\infty  e^{-\lambda_n\int\limits_0^ta(\tau)d\tau} \int\limits_0^t e^{\lambda_n\int\limits_0^sa(\tau)d\tau}f_n(s)ds(q(x)-\lambda_n)\phi_n(x)\right|^2dx,
\end{eqnarray*}
and using \eqref{est2.4}, \eqref{I2}, \eqref{f'} we get the estimates
\begin{eqnarray*}
\|\partial^2_xu(t,\cdot)\|^2_{L^2}&\lesssim& \|q\|^2_{L^\infty} \left(\|u_0\|^2_{L^2}+T^2\|f\|^2_{C([0,T],L^2(0,1))}\right)\\
&+&\|u_0\|^2_{W^2_\mathcal{L}}+\frac{2T^2}{a_0^2}\|f\|^2_{C^1([0,T],L^2(0,1))}.
\end{eqnarray*}
This completes the proof.
\end{proof}

As in the homogeneous case, for the non-homogeneous case we also eliminate the dependence of Sobolev spaces on the operator $\mathcal{L}$, while the estimates will depend only on the initial data $u_0$ and potential $q$.

\begin{cor}\label{cor2}
Assume that $q\in L^2(0,1)$, $a(t)\geq a_0>0$ for all $t\in [0,T]$, $a'\in L^\infty[0,T]$ and $f\in C^1([0,T],L^2(0,1))$. If the initial condition satisfies $u_0 \in L^2(0,1)$ and $u_0''\in L^2(0,1)$, then the non-homogeneous heat equation with initial/boundary conditions  \eqref{nonh} has unique solution $u\in C([0,T], L^2(0,1))$ such that
\begin{equation}\label{ec-nh1}
    \|u(t,\cdot)\|_{L^2}\lesssim \|u_0\|_{L^2}+T\|f\|_{C([0,1],L^2(0,1))},
\end{equation}
\begin{equation}\label{ec-nh2}
\|\partial_t u(t,\cdot)\|_{L^2}\lesssim \|a\|_{L^\infty[0,T]}\left(\|u''_0\|_{L^2}+\|q\|_{L^\infty}\|u_0\|_{L^2}+\frac{T}{a_0}\|f\|_{C^1([0,T],L^2(0,1))}\right),
\end{equation}
\begin{eqnarray}\label{ec-nh3}
\|\partial_x u(t,\cdot)\|_{L^2}
&\lesssim& \left(\|u''_0\|_{L^2}+\|q\|_{L^\infty}\|u_0\|_{L^2}\right)\left(1 +\|\nu\|_{L^2}\right)\nonumber\\
&+&\frac{T}{a_0}\left(1 +\|\nu\|_{L^2}\right)\|f\|_{C^1([0,T],L^2(0,1))}\nonumber\\
&+&\|\nu\|_{L^\infty}\left(\|u_0\|_{L^2}+T\|f\|_{C([0,T],L^2(0,1))}\right),
\end{eqnarray}
\begin{eqnarray}\label{ec-nh4}
\left\|\partial_x^2u(t, \cdot)\right\|_{L^2}&\lesssim&\|u''_0\|_{L^2}+\frac{T}{a_0}\|f\|_{C^1([0,T],L^2(0,1))}\nonumber\\
&+&\|q\|_{L^\infty}\left(\|u_0\|_{L^2}+T\|f\|_{C([0,T],L^2(0,1))}\right),
\end{eqnarray}
where the constants in these inequalities are independent of $u_0$, $q$, $a$ and $f$.
\end{cor}

The proof of Corollary \ref{cor2} immediately follows from Corollary \ref{cor1} and the proof of Theorem \ref{non-hom}.

Note that the main difference between the estimates obtained in the homogeneous case and those obtained in the inhomogeneous case is that the latter include additional conditions on the function in the right-hand side. Both cases - both homogeneous and inhomogeneous - are important, and these estimates will be used in the future to establish the well-posedness of very weak solutions of the equation with singularities.

\section{Very weak solutions}

In this section we analyse solutions for the less regular coefficients $q$, $a$ and the initial condition $u_0$. To obtain the well-posedness in such cases, we will use the concept of very weak solutions.

Examples of coefficients that we can allow here are e.g. $a=\delta_1$, $q=\delta_{\frac{1}{2}}$, the delta-distributions in $t$ and $x$, respectively, or their combinations.

Suppose that the coefficient $q$ and the initial condition $u_0$ are distributions on $(0,1)$, the coefficient $a$ is distribution on $[0,T]$. To regularize distributions, we introduce the following definition.

\begin{defn}\label{D1}  (i)
A net of functions $\left(u_\varepsilon=u_\varepsilon(t,x)\right)$ is said to be uniformly $L^2$-moderate if there exist $N\in \mathbb{N}_0$ and $C>0$ such that 
$$\|u_\varepsilon(t,\cdot)\|_{L^2}\leq C \varepsilon^{-N}, \quad \text{for all } t\in[0,T].$$
(ii) A net of functions ($u_{0,\varepsilon}=u_{0,\varepsilon}(x)$) is said to be $H^2$-moderate if there exist $N\in \mathbb{N}_0$ and $C>0$ such that
$$\|u_{0,\varepsilon}\|_{L^2}\leq C\varepsilon^{-N},\qquad \|u''_{0,\varepsilon}\|_{L^2}\leq C\varepsilon^{-N}.$$
\end{defn}

\begin{defn}\label{D11} (i) A net of functions $\left(q_\varepsilon=q_\varepsilon(x)\right)$ is said to be $L^\infty$-moderate if there exist $N\in \mathbb{N}_0$ and $C>0$ such that 
$$\|q_\varepsilon\|_{L^\infty(0,1)}\leq C \varepsilon^{-N}.$$
(ii) A net of functions $(a_\varepsilon=a_\varepsilon
(t))$ is said to be $L^\infty$-moderate if there exist $N\in \mathbb{N}_0$ and $C>0$ such that
$$\|a_\varepsilon\|_{L^\infty[0,T]}\leq C \varepsilon^{-N}.$$

\end{defn}

\begin{rem} We note that such assumptions are natural for distributional coefficients in the sense that regularisations of distributions are moderate. Precisely, by the structure theorems for distributions (see, e.g. \cite{Friedlander}), we know that distributions 
\begin{equation}\label{moder}
  \mathcal{D}'(0,1) \subset \{L^\infty(0,1) -\text{moderate families} \},  
\end{equation}
and we see from \eqref{moder}, that a solution to an initial/boundary problem may not exist in the sense of distributions, while it may exist in the set of $L^\infty$-moderate functions. 
\end{rem}

As an illustration, let us consider $f\in L^2(0,1)$, $f:(0,1)\to \mathbb{C}$. We introduce the function 
$$\Tilde{f}=\left\{\begin{array}{l}
    f, \text{ on }(0,1),  \\
    0,  \text{ on }\mathbb{R} \setminus (0,1),
\end{array}\right.$$
then $\Tilde{f}:\mathbb{R}\to \mathbb{C}$, and $\Tilde{f}\in \mathcal{E}'(\mathbb{R}).$

Let $\Tilde{f}_\varepsilon=\Tilde{f}*\psi_\varepsilon$ be obtained as the convolution of $\Tilde{f}$ with a Friedrich mollifier $\psi_\varepsilon$, where 
$$\psi_\varepsilon(x)=\frac{1}{\varepsilon}\psi\left(\frac{x}{\varepsilon}\right),\quad \text{for}\,\, \psi\in C^\infty_0(\mathbb{R}),\, \int \psi=1. $$
Then the regularising net $(\Tilde{f}_\varepsilon)$ is $L^p$-moderate for any $p \in [1,\infty)$, and it approximates $f$ on $(0,1)$:
$$0\leftarrow \|\Tilde{f}_\varepsilon-\Tilde{f}\|^p_{L^p(\mathbb{R})}\approx \|\Tilde{f}_\varepsilon-f\|^p_{L^p(0,1)}+\|\Tilde{f}_\varepsilon\|^p_{L^p(\mathbb{R}\setminus (0,1))}.$$
Now, let us introduce the concept of a very weak solution to the initial/boundary value problem \eqref{C.p1}-\eqref{C.p3}.

\begin{defn}\label{D2}
Let $q\in \mathcal{D}'(0,1)$, $a\in \mathcal{D}'[0,T]$.  The net $(u_\varepsilon)_{\varepsilon>0}$ is said to be a very weak solution to the initial/boundary problem  \eqref{C.p1}-\eqref{C.p3} if there exists an $L^\infty$-moderate regularisation $q_\varepsilon$ of $q$, an $L^\infty$-moderate regularisation $a_\varepsilon$ of $a$, and an $H^2$-moderate regularisation $u_{0,\varepsilon}$ of $u_0$, such that
\begin{equation}\label{vw1}
         \left\{\begin{array}{l}\partial_t u_\varepsilon(t,x)+a_\varepsilon(t)\left(-\partial^2_x u_\varepsilon(t,x)+q_\varepsilon(x) u_\varepsilon(t,x)\right)=0,\quad (t,x)\in [0,T]\times(0,1),\\
 u_\varepsilon(0,x)=u_{0,\varepsilon}(x),\,\,\, x\in (0,1), \\
u_\varepsilon(t,0)=0=u_\varepsilon(t,1), \quad t\in[0,T],
\end{array}\right.\end{equation}
and $(u_\varepsilon)$ and $(\partial_t u_\varepsilon)$ are uniformly $L^{2}$-moderate.
\end{defn}

Thus we obtain the following properties of very weak solutions.

\begin{thm}[Existence]\label{Ext}
Let the coefficient $q$ and initial condition $u_0$ be distributions in $(0,1)$, the coefficient $a$ be a distribution in $[0,T]$. Then the initial/boundary problem  \eqref{C.p1}-\eqref{C.p3} has a very weak solution.
\end{thm}
\begin{proof}
Because the formulation of \eqref{C.p1}-\eqref{C.p3} in this instance might be impossible within the distributional sense due to complications arising from the product of distributions, we replace \eqref{C.p1}-\eqref{C.p3} with a regularized equation. In simpler terms, we introduce regularisations for $q$, $a$ and $u_0$ by using corresponding sets of smooth functions, denoted as $q_\varepsilon$, $a_\varepsilon$ and $u_{0,\varepsilon}$, derived from $C^\infty(0,1)$.

Therefore, $q_\varepsilon$, $a_\varepsilon$  are $L^\infty$-moderate regularisations and $u_{0,\varepsilon}$ be $H^2$-moderate regularisation of the coefficients $q$, $a$ and the Cauchy condition $u_0$ respectively. So by Definition 1 there exist $N\in \mathbb{N}_0$ and $C_1>0,\,C_2>0,\, C_3>0, \, C_4>0$, such that
$$\|q_\varepsilon\|_{L^\infty}\leq C_1\varepsilon^{-N},\quad \|u_{0,\varepsilon}\|_{L^2}\leq C_2\varepsilon^{-N}, \quad \|u''_{0,\varepsilon}\|_{L^2}\leq C_3\varepsilon^{-N} ,\quad \|a_\varepsilon\|_{L^\infty}\leq C_4 \varepsilon^{-N}.$$

Next, we consider the regularised problem \eqref{vw1}, fixing $\varepsilon\in (0,1]$. In this case, all statements and calculations of Theorem \ref{th1} are valid. Hence, by Theorem \ref{th1}, the equation \eqref {vw1} has a unique solution $u_\varepsilon(t,x)$ in the space $C([0,T];L^2(0,1))$.

By Corollary \ref{cor1}, there exist $N\in \mathbb{N}_0$ and $C>0$, such that
$$\|u_\varepsilon(t,\cdot)\|_{L^2}\lesssim \|u_{0,\varepsilon}\|_{L^2}\leq C\varepsilon^{-N},$$
$$\|\partial_t u_\varepsilon(t,\cdot)\|_{L^2}\lesssim \|a_\varepsilon\|^2_{L^\infty[0,T]}\left(\|u''_{0,\varepsilon}\|_{L^2}+\|q_\varepsilon\|_{L^\infty}\|u_{0, \varepsilon}\|_{L^2}\right)\leq C\varepsilon^{-N},$$
where the constants in these inequalities are independent of $u_0$, $q$ and $a$.
Hence, $(u_\varepsilon)$ is moderate, and the proof of Theorem \ref{Ext} is complete.
\end{proof}

Explaining the uniqueness of very weak solutions involves the process of \textquotedblleft measuring\textquotedblright \, changes in the associated families: the conditions of \textquotedblleft insignificance\textquotedblright \,  for these families of functions/distributions can be expressed as follows:

\begin{defn}[Negligibility]\label{D3}
Let $(u_\varepsilon)$, $(\Tilde{u}_\varepsilon)$ be two nets in $L^2(0,1)$. Then, the net $(u_\varepsilon-\Tilde{u}_\varepsilon)$ is called $L^2$-negligible, if for every $N\in \mathbb{N}$ there exist $C>0$ such that the following condition is satisfied
$$\|u_\varepsilon-\Tilde{u}_\varepsilon\|_{L^2}\leq C \varepsilon^N,$$
for all $\varepsilon\in (0,1]$. In the case where $u_\varepsilon=u_\varepsilon(t,x)$ is a net depending on $t\in [0,T]$, then the negligibility condition can be described as 
$$\sup\limits_{t\in[0,T]}\|u_\varepsilon(t,\cdot)-\Tilde{u}_\varepsilon(t,\cdot)\|_{L^2}\leq C \varepsilon^N,$$
uniformly in $t\in [0,T]$. The constant $C$ can depends on $N$ but not on $\varepsilon$. We say that $u_\varepsilon-\Tilde{u}_\varepsilon$ is then uniformly $L^2$-negligible.
\end{defn}

Let us state the  \textquotedblleft$\varepsilon$-parameterised problems\textquotedblright \,to be considered:
\begin{equation}\label{un1}
    \left\{\begin{array}{l}\partial_t u_\varepsilon(t,x)+a_\varepsilon(t)\left(-\partial^2_x u_\varepsilon(t,x)+q_\varepsilon(x) u_\varepsilon(t,x)\right)=0,\quad (t,x)\in [0,T]\times(0,1),\\
 u_\varepsilon(0,x)=u_{0,\varepsilon}(x),\,\,\, x\in (0,1), \\
u_\varepsilon(t,0)=0=u_\varepsilon(t,1),\quad t\in[0,T],
\end{array}\right.
\end{equation}
and
\begin{equation}\label{un2}
    \left\{\begin{array}{l}\partial_t \Tilde{u}_\varepsilon(t,x)+\Tilde{a}_\varepsilon(t)\left(-\partial^2_x \Tilde{u}_\varepsilon(t,x)+\Tilde{q}_\varepsilon(x) \Tilde{u}_\varepsilon(t,x)\right)=0,\quad (t,x)\in [0,T]\times(0,1),\\
 \Tilde{u}_\varepsilon(0,x)=\Tilde{u}_{0,\varepsilon}(x),\,\,\, x\in (0,1), \\
\Tilde{u}_\varepsilon(t,0)=0=\Tilde{u}_\varepsilon(t,1), \quad t\in[0,T].
\end{array}\right.
\end{equation}

\begin{defn}[Uniqueness of the very weak solution]\label{D4}
Let $q\in \mathcal{D}'(0,1)$, $a\in \mathcal{D}'[0,T]$. We say that initial/boundary problem \eqref{C.p1}-\eqref{C.p3} has an unique very weak solution, if
for all $L^\infty$-moderate nets $q_\varepsilon$, $\Tilde{q}_\varepsilon$, such that $(q_\varepsilon-\Tilde{q}_\varepsilon)$ is $L^\infty$-negligible; for all $L^\infty$-moderate nets $a_\varepsilon$, $\Tilde{a}_\varepsilon$, such that $(a_\varepsilon-\Tilde{a}_\varepsilon)$ is $L^\infty$-negligible; and for all $H^2$-moderate regularisations $u_{0,\varepsilon},\,\Tilde{u}_{0,\varepsilon}$, such that $(u_{0,\varepsilon}-\Tilde{u}_{0,\varepsilon})$, is $L^2$-negligible, we have that $u_\varepsilon-\Tilde{u}_\varepsilon$ is uniformly $L^2$-negligible.
\end{defn}

\begin{thm}[Uniqueness of the very weak solution]\label{Th-U}
Let the coefficient $q$ and initial condition $u_0$ be distributions in $(0,1)$, and the coefficient $a$
be a distribution in $[0, T]$. Then the very weak solution to the initial/boundary problem  \eqref{C.p1}-\eqref{C.p3} is unique.
\end{thm}
\begin{proof}
We denote to the family of solutions for the initial/boundary problems \eqref{un1} and \eqref{un2} as $u_\varepsilon$ and $\Tilde{u}_\varepsilon$ respectively. Defining $U_\varepsilon$ as the difference between these nets, specifically, $U_\varepsilon:=u_\varepsilon(t,\cdot)-\Tilde{u}_\varepsilon(t,\cdot)$, we can ascertain that $U_\varepsilon$ serves as a solution to the following equation
    \begin{equation}\label{unq}
    \left\{\begin{array}{l}\partial_t U_\varepsilon(t,x)+a_\varepsilon(t)\left(-\partial^2_x U_\varepsilon(t,x)+q_\varepsilon(x) U_\varepsilon(t,x)\right)=f_\varepsilon(t,x),\quad (t,x)\in [0,T]\times(0,1),\\
 U_\varepsilon(0,x)=(u_{0,\varepsilon}-\Tilde{u}_{0,\varepsilon})(x),\,\,\, x\in (0,1), \\
U_\varepsilon(t,0)=0=U_\varepsilon(t,1),
\end{array}\right.
\end{equation}
where we set 
\begin{eqnarray*}
    f_\varepsilon(t,x)&:=&(a_\varepsilon(t)-\Tilde{a}_\varepsilon(t))\partial^2_x\Tilde{u}_\varepsilon(t,x)\\
    &+&\left\{\Tilde{a}_\varepsilon(t)\left(\Tilde{q}_\varepsilon(x)-q_\varepsilon(x)\right)+q_\varepsilon(x)\left(\Tilde{a}_\varepsilon(x)-a_\varepsilon(x)\right)\right\}\Tilde{u}_\varepsilon(t,x)
\end{eqnarray*}
for the mass term to the non-homogeneous initial/boundary problem \eqref{unq}.

Passing to the $L^2$-norm of the $U_\varepsilon$ and using \eqref{ec-nh1}, we obtain the following result
$$
\|U_\varepsilon(t,\cdot)\|^2_{L^2}\lesssim \|U_\varepsilon(0,\cdot)\|^2_{L^2}+T^2\|f_\varepsilon\|^2_{C([0,T],L^2(0,1))}.
$$
For the $\|f_\varepsilon\|^2_{C([0,T],L^2(0,1))}$ by using \eqref{25}, \eqref{ec4} we get
\begin{eqnarray*}
\|f_\varepsilon\|^2_{C([0,T],L^2(0,1))}&\lesssim& \|\Tilde{a}_\varepsilon-a_\varepsilon\|^2_{L^\infty[0,T]}\left(\|\Tilde{u}''_{0,\varepsilon}\|^2_{L^2}+2\|\Tilde{q}_\varepsilon\|^2_{L^\infty}\|\Tilde{u}_{0,\varepsilon}\|^2_{L^2}\right)\\
&+&\|\Tilde{q}_\varepsilon-q_\varepsilon\|^2_{L^\infty}\|\Tilde{a}_\varepsilon\|^2_{L^\infty[0,T]}\|\Tilde{u}_\varepsilon\|^2_{C([0,T],L^2(0,1))}\\
&+&\|\Tilde{a}_\varepsilon-a_\varepsilon\|^2_{L^\infty[0,T]}\|q_\varepsilon\|^2_{L^\infty}\|\Tilde{u}_\varepsilon\|^2_{C([0,T],L^2(0,1))}.
\end{eqnarray*}
Next, using the initial condition of \eqref{unq} we obtain
\begin{eqnarray*}
\|U_\varepsilon(t,\cdot)\|^2_{L^2}&\lesssim& \|u_{0,\varepsilon}-\Tilde{u}_{0,\varepsilon}\|^2_{L^2}+T^2\|\Tilde{a}_\varepsilon-a_\varepsilon\|^2_{L^\infty[0,T]}\left(\|\Tilde{u}''_{0,\varepsilon}\|^2_{L^2}+2\|\Tilde{q}_\varepsilon\|^2_{L^\infty}\|\Tilde{u}_{0,\varepsilon}\|^2_{L^2}\right)\\
&+&T^2\|\Tilde{q}_\varepsilon-q_\varepsilon\|^2_{L^\infty}\|\Tilde{a}_\varepsilon\|^2_{L^\infty[0,T]}\|\Tilde{u}_\varepsilon\|^2_{C([0,T],L^2(0,1))}\\
&+&T^2\|\Tilde{a}_\varepsilon-a_\varepsilon\|^2_{L^\infty[0,T]}\|q_\varepsilon\|^2_{L^\infty}\|\Tilde{u}_\varepsilon\|^2_{C([0,T],L^2(0,1))}.
\end{eqnarray*}
Taking into account the negligibility of the nets $u_{0,\varepsilon}-\Tilde{u}_{0,\varepsilon}$, $q_\varepsilon-\Tilde{q}_\varepsilon$ and $a_\varepsilon-\Tilde{a}_\varepsilon$ we get
$$\|U_\varepsilon(t,\cdot)\|^2_{L^2}\leq C_1\varepsilon^{N_1}+\varepsilon^{N_2}\left(C_2\varepsilon^{-N_3}+C_3\varepsilon^{-N_4}\right)+\varepsilon^{N_5}\left(C_4\varepsilon^{-N_6}+C_5\varepsilon^{-N_7}\right)$$
for some $C_1>0,\,C_2>0, \, C_3>0,\,C_4>0,\, C_5>0,\,\,N_3,\,N_4,\, N_6,\, N_7\in \mathbb{N}$ and all $N_1,\,N_2,\, N_5 \in \mathbb{N}$, since $\Tilde{u}_\varepsilon$ is moderate. Then, for some $C_M>0$ and all $M\in \mathbb{N}$
$$\|U_\varepsilon(t,\cdot)\|_{L^2}\leq C_M \varepsilon^M.$$
The last estimate holds true uniformly in $t$ , and this completes the proof of Theorem \ref{Th-U}.
\end{proof}

\begin{thm}[Consistency]\label{Th-C} Assume that $q\in L^\infty(0,1)$, $a\in L^\infty[0,T]$ and let $(q_\varepsilon)$ be any $L^\infty$-regularisation of $q$, $(a_\varepsilon)$ be any $L^\infty$-regularisation of $a$, that is $\|q_\varepsilon-q\|_{L^\infty}\to 0$, $\|a_\varepsilon-a\|_{L^\infty[0,T]}\to 0$ as $\varepsilon\to 0$. Let the initial condition satisfy $u_0, \, u''_0 \in L^2(0,1)$. Let $u$ be a very weak solution of the initial/boundary problem \eqref{C.p1}-\eqref{C.p3}. Then for any families $q_\varepsilon$, $a_\varepsilon$, $u_{0,\varepsilon}$ such that $\|u_{0}-u_{0,\varepsilon}\|_{L^2}\to 0$, $\|q-q_{\varepsilon}\|_{L^\infty}\to 0$, $\|a-a_\varepsilon\|_{L^\infty[0,T]}\to 0$ as $\varepsilon\to 0$, any representative $(u_\varepsilon)$ of the very weak solution converges as 
$$\sup\limits_{0\leq t\leq T}\|u(t,\cdot)-u_\varepsilon(t,\cdot)\|_{L^2}\to 0$$ 
for $\varepsilon\to 0$ to the unique classical solution $u\in C([0,T];L^2(0,1))$ of the initial/boundary problem \eqref{C.p1}-\eqref{C.p3} given by Theorem \ref{th1}.
\end{thm}

\begin{proof}
As in our assumption, for $u$ and for $u_\varepsilon$ we introduce an auxiliary notation $V_\varepsilon(t, x):= u(t,x)-u_\varepsilon(t,x)$. Then the net $V_\varepsilon$ is a solution to the initial/boundary problem
\begin{equation}\label{51}
    \left\{\begin{array}{l}
        \partial_tV_\varepsilon(t,x)+a_\varepsilon(t)\left(-\partial^2_xV_\varepsilon(t,x)+q_\varepsilon(x)V_\varepsilon(t,x)\right)=f_\varepsilon(t,x),\\
        V_\varepsilon(0,x)=(u_0-u_{0,\varepsilon})(x),\quad x\in (0,1),\\        V_\varepsilon(t,0)=0=V_\varepsilon(t,1), \quad t\in[0,T],
    \end{array}\right.
\end{equation}
where 
$$f_\varepsilon(t,x):=(a(t)-a_\varepsilon(t))\partial^2_x u(t,x)+\left\{a_\varepsilon(t)(q_\varepsilon(x)-q(x))+q(x)(a_\varepsilon(t)-a(t))\right\}u(t,x).$$  Analogously to Theorem \ref{Th-U} we have that
\begin{eqnarray*}
    \|V_\varepsilon(t,\cdot)\|^2_{L^2}&\lesssim& \|u_{0}-u_{0,\varepsilon}\|^2_{L^2}+T^2\|a-a_\varepsilon\|^2_{L^\infty[0,T]}\left(\|u''_0\|^2_{L^2}+2\|q\|^2_{L^\infty}\|u_0\|^2_{L^2}\right)\\
&+&T^2\|a-a_\varepsilon\|^2_{L^\infty[0,T]}\|q\|^2_{L^\infty}\|u\|^2_{C([0,T],L^2(0,1))}\\
&+&T^2\|q-q_\varepsilon\|^2_{L^\infty}\|a_\varepsilon\|^2_{L^\infty[0,T]}\|u\|^2_{C([0,T],L^2(0,1))}.
\end{eqnarray*}
$$$$
Since
$$\|u_{0}-{u}_{0,\varepsilon}\|_{L^2}\to 0,\quad  \|q_\varepsilon-q\|_{L^\infty}\to 0, \quad \|a-a_\varepsilon\|_{L^\infty[0,T]}\to 0$$
for $\varepsilon\to 0$ and $u$ is the classical solution of the initial/boundary problem \eqref{C.p1}-\eqref{C.p3} we get
$$\|V_\varepsilon(t,\cdot)\|_{L^2}\to 0$$
for $\varepsilon\to 0$. This proves Theorem \ref{Th-C}.
\end{proof}

\textbf{Conflict of interest statement:} there is no conflict of interest.


\begin{thebibliography}{References}
\bibitem{Al-G} Al-Gwaiz, M. A. Sturm-Liouville Theory and its Applications. Springer-Verlag London, 2008. https://doi.org/10.1007/978-1-84628-972-9
\bibitem{ARST1} Altybay, A., Ruzhansky, M., Sebih, M. E., Tokmagambetov, N. The heat equation with strongly singular potentials. Appl. Math. Comput. 399, 126006 (2021). doi:10.1016/j.amc.2021.126006
\bibitem{ARST2} Altybay, A., Ruzhansky, M., Sebih, M. E. and Tokmagambetov, N. Fractional Klein-Gordon equation with singular mass, Chaos Solitons \& Fractals 143 (2021), 110579. https://doi.org/10.1016/j.chaos.2020.110579
\bibitem{ARST3} Altybay, A., Ruzhansky, M., Sebih, M. E. and Tokmagambetov, N. Fractional Schrödinger equations with singular potentials of higher-order, Rep. Math. Phys. 87, 129 (2021). DOI: 10.1016/S0034-4877(21)00016-1
\bibitem{CRT1} Chatzakou, M., Ruzhansky, M., Tokmagambetov, N. Fractional Schrödinger equations with singular potentials of higher order. II: Hypoelliptic case. Rep. Math. Phys. 89, 59–79 (2022). https://doi.org/10.1016/S0034-4877(22)00010-6
\bibitem{CRT2} Chatzakou, M., Ruzhansky, M., Tokmagambetov, N. Fractional Klein-Gordon equation with singular mass. II: Hypoelliptic case. Complex Var. Elliptic Equ., 67 (2022). https://doi.org/10.1080/17476933.2021.1950146
\bibitem{CRT0} Chatzakou, M., Ruzhansky, M., Tokmagambetov, N. The Heat Equation with Singular Potentials. II: Hypoelliptic Case. Acta Appl. Math. 179, 2 (2022). https://doi.org/10.1007/s10440-022-00487-w
\bibitem{Friedlander} Friedlander, F. G., Joshi, M. Introduction to the Theory of Distributions, Cambridge University Press, 1998.
\bibitem{Gar-Ruz} Garetto, C., Ruzhansky, M. Hyperbolic second order equations with non-regular time dependent coefficients. Arch. Rational Mech. Anal., 217, 113–154 (2015). https://doi.org/10.1007/s00205-014-0830-1
\bibitem{Separ} Nandakumaran, A.K., Datti, P.S. Partial Differential Equations: Classical Theory with a Modern Touch. Cambridge IISc Series. Cambridge University Press, 2020. https://doi.org/10.1017/9781108839808
\bibitem{Ince} Ince, E. L. Ordinary differential equations. 2nd ed. New York: Dover Publ., 1956.
\bibitem{N-zSk} Neiman-zade, M. I., Shkalikov, A. A. Schrödinger operators with singular potentials from the space of multiplicators. Math Notes 66, 599–607 (1999). https://doi.org/10.1007/BF02674201
\bibitem{Lan-Ham} Ruzhansky, M., Tokmagambetov, N. Very weak solutions of wave equation for Landau Hamiltonian with irregular electromagnetic field. Lett. Math. Phys. 107, 591–618 (2017). https://doi.org/10.1007/s11005-016-0919-6
\bibitem{Ruz-Tok} Ruzhansky M., Tokmagambetov N., Wave equation for operators with discrete spectrum and irregular propagation speed. Arch. Ration. Mech. Anal., 226 (2017), 1161-1207. https://doi.org/10.1007/s00205-017-1152-x
\bibitem{RYS} Ruzhansky, M., Shaimardan, S., Yeskermessuly, A. Wave equation for Sturm-Liouville operator with singular potentials. J. Math. Anal. Appl., 531, 1, 2, (2024), 127783. https://doi.org/10.1016/j.jmaa.2023.127783.
\bibitem{R-Y} Ruzhansky, M., Yessirkegenov, N. Very weak solutions to hypoelliptic wave equations. J. Diff. Equ. 268, 2063 (2020). https://doi.org/10.1016/j.jde.2019.09.020
\bibitem{RYes} Ruzhansky, M., Yeskermessuly, A. Wave equation for Sturm–Liouville operator with singular intermediate coefficient and potential. Bull. Malays. Math. Sci. Soc. 46, 195 (2023). https://doi.org/10.1007/s40840-023-01587-y 
\bibitem{Savch} Savchuk, A. M. On the eigenvalues and eigenfunctions of the Sturm-Liouville operator with a singular potential. Math. Notes, vol. 69, No. 2, 2001, 245-252.  https://doi.org/10.1023/A:1002880520696
\bibitem{Sav-Shk} Savchuk, A. M., Shkalikov, A. A. Sturm-Liouville operators with singular potentials. Math Notes 66, 741–753 (1999). https://doi.org/10.1007/BF02674332
\bibitem{Sav-Shk2} Savchuk, A. M., Shkalikov, A. A. On the eigenvalues of the Sturm-Liouville operator with potentials from Sobolev spaces. Math Notes 80, 814–832 (2006). https://doi.org/10.1007/s11006-006-0204-6
\bibitem{SV} Shkalikov, A. A., Vladykina, V. E. Asymptotics of the solutions of the Sturm–Liouville equation with singular coefficients. Math Notes 98, 891–899 (2015). https://doi.org/10.1134/S0001434615110218
\end{thebibliography}
\end{document}